\documentclass{article}
\usepackage[utf8]{inputenc}
\usepackage[margin=1.2in]{geometry} 

\usepackage{amssymb}
\usepackage{mathtools}
\usepackage[autostyle=true]{csquotes}
\usepackage{mathrsfs}
\usepackage{graphicx}
\usepackage{dsfont}
\usepackage{subcaption}
\usepackage{appendix}
\usepackage{stmaryrd}
\usepackage{hyperref}
\usepackage{enumitem}
\usepackage{accents}
\usepackage{framed}
\usepackage[amsmath,thmmarks,framed,hyperref]{ntheorem}
\usepackage{float}
\usepackage{color}
\usepackage{svg}
\usepackage[capitalise]{cleveref}
\usepackage[framemethod=TikZ]{mdframed}

\hypersetup{
    colorlinks=true,
    linkcolor=blue,
    citecolor = red,
    filecolor=magenta,      
    urlcolor=cyan
}

\newcommand{\R}{\mathbb{R}}
\newcommand{\Z}{\mathbb{Z}}
\newcommand{\N}{\mathbb{N}}
\newcommand{\Q}{\mathbb{Q}}

\newcommand{\A}{\mathbb{A}}
\newcommand{\Sp}{\mathbb{S}}
\newcommand{\Di}{\mathbb{D}}
\newcommand{\M}{\mathbb{M}}
\newcommand{\I}{\mathbb{I}}
\newcommand{\T}{\mathbb{T}}

\newcommand{\croset}[1]{\left\lbrace #1 \right\rbrace}

\newcommand{\symp}[2]{\mathrm{Symp}^{#1}(#2)}
\newcommand{\sympo}[2]{\mathrm{Symp}_\diamond^{#1}(#2)}
\newcommand{\sympc}[2]{\mathrm{Symp}_\star^{#1}(#2)}

\newcommand{\isum}[2]{\displaystyle\sum_{#1}^{#2}}
\newcommand{\inint}[2]{\displaystyle\int_{#1}^{#2}}
\newcommand{\lconv}[2]{\underset{#1 \rightarrow #2}{\lim}}
\newcommand{\rconv}[2]{\underset{#1 \rightarrow #2}{\longrightarrow}}

\newcommand{\iconj}[2]{#1\circ #2 \circ #1^{-1}}
\newcommand{\lequiv}[2]{\underset{#1 \rightarrow #2}{\sim}}

\newcommand{\rest}[2]{{
  #1 
  |_{#2} 
  }}

\DeclareMathOperator*{\diam}{diam}
\DeclareMathOperator{\Leb}{Leb}

\DeclareMathOperator*{\id}{id}

\newcommand{\mesy}[1]{\Leb_{\T \times \croset{#1}}}

\graphicspath{ {./images/} }

\theorempreskip{0.2cm}
\theorempostskip{0.2cm}
\theoremstyle{plain}
\theoremseparator{.}
\newtheorem{theorem}{Theorem}[section]
\newtheorem{lemma}[theorem]{Lemma}
\newtheorem{prop}[theorem]{Proposition}
\crefname{prop}{Proposition}{Propositions}
\newtheorem{coro}[theorem]{Corollary}
\newtheorem{defin}[theorem]{Definition}
\newtheorem{remark}[theorem]{Remark}
\newtheorem{fact}[theorem]{Fact}

\newtheorem*{theorem-no}{Theorem}
\crefname{equation}{equation}{equations}
\crefname{fact}{Fact}{Facts}

\theoreminframepreskip{0cm}
\theoreminframepostskip{0cm}
\newtheorem{mainthm}{Theorem}

\crefname{mainthm}{Theorem}{Theorems}

\renewenvironment{theorem-no}[1][Theorem]{%
\vspace{10pt}
  \par\noindent\textbf{Theorem #1.}\enspace  
  \itshape
}{\par 
\vspace{10pt} \normalfont}

\theoremstyle{nonumberplain}
\theoremheaderfont{\itshape}
\theorembodyfont{\normalfont}
\theoremsymbol{\ensuremath{\square}}
\newtheorem{proof}{Proof}

\numberwithin{equation}{section}

\renewenvironment{proof}[1][Proof]{%
\vspace{0.1cm}
  \par\noindent\textit{#1.}\enspace  
}{%
  \hfill\ensuremath{\square}\par  
  \vspace{0.2cm}
}

\title{Analytic pseudo-rotations displaying ergodicity or emergence on the sphere and disk}
\author{Yann Delaporte\footnote{IMJ-PRG, CNRS, Sorbonne University, Université Paris Cité, partially supported by the ERC project 818737 Emergence of wild differentiable dynamical systems.}}

\begin{document}

\maketitle
\date{}

\begin{abstract}

We construct analytic pseudo-rotations on the sphere and the disk with specific properties. We obtain analytic pseudo-rotations which are ergodic. Then, in opposition to ergodicity, we construct analytic pseudo-rotations which exhibit a maximum order of local emergence. To achieve this, we apply a principle introduced by Berger, based on the Approximation by Conjugacy method of Anosov-Katok. 

\end{abstract}

\tableofcontents

\section{Introduction}

In dynamical systems, one of the many challenges is to prove that properties can be satisfied by very regular dynamics. To this end, Anosov and Katok introduced the Approximation by Conjugacy (AbC) method \cite{anosov_new_1970}, which gives the existence of smooth, ergodic, and volume-preserving dynamics on a compact manifold endowed with a $\R/\Z$ action, such as the cylinder, the sphere or the disk. A particularity of the AbC method is that it produces \emph{pseudo-rotations} which are symplectomorphisms (area and orientation preserving maps) of surfaces with a finite number of periodic points (see \cite{beguin_pseudo-rotations_2006}).\\

This method has been highly effective in achieving smoothness for several properties, including minimal ergodicity, minimal set with a given measure, maximum local emergence, transitivity without ergodicity, among others, as presented in the survey \cite{fayad_constructions_2004} or in \cite{berger_analytic_2022}.\\

However it is not obvious how one can adapt the AbC method to the analytic setting. In \cite[§7.1]{fayad_constructions_2004} and \cite{herman_open_1998}, it is pointed out that the difficulties to obtain analyticity with the AbC method is due to the shrinking of the convergence radius when we compose conjugacies. Some analytic examples have been obtained on tori with \enquote{block-slide} map and Fourier series in \cite{banerjee_real-analytic_2019} or with a straightforward analytic construction in \cite[Theorem 2.1]{furstenberg_strict_1961}. Also, Fayad and Katok managed to apply the AbC method on odd spheres and obtained analytic uniquely ergodic volume preserving maps, see \cite{fayad_analytic_2014}.\\

Recently, the case of the cylinder has been solved in \cite{berger_analytic_2022} where Berger uses entire functions to face the issue of the convergence radius. His construction enabled to obtain   entire pseudo-rotations, each of which leaves  invariant a cylinder with a non-straight boundary. Then the boundary is  easily straightened by using an analytic conjugacy. As a result, Berger disproved a conjecture of Birkhoff \cite{birkhoff_unsolved_1941}
by showing the existence of analytic pseudo-rotations of the cylinder which are ergodic. In addition, he shows to the contrary the existence of analytic pseudo-rotations which exhibit local emergence of maximal order (see definition Page \pageref{def:emer}).\\

Later on, to address the question of the sphere and the disk, Berger introduced the \emph{AbC Principle}, which allows to perform analytic construction with the AbC method on surfaces such as the disk, the cylinder and the sphere \cite{berger_analytic_2024}. Basically, the Principle states that if a property $(\mathcal{P})$ is realizable by an \emph{AbC scheme} in a $C^r$-topology, for $0 \leq r \leq \infty$, then there exists an analytic symplectomorphism satisfying $(\mathcal{P})$. See \cref{def:abc} for the AbC scheme's definition. To prove the Principle, Berger uses deformation of the analytic and symplectic structure of these surfaces during the construction of the symplectomorphism. This work enabled him to obtain analytic pseudo-rotations which are transitive on the sphere and the disk.
\\

As problematized in \cite{berger_analytic_2024}, it is also possible to obtain ergodicity or to the contrary maximum local emergence with his principle, which are the two main results.

\begin{mainthm}\label{thm:erg_scheme}
There exist analytic ergodic pseudo-rotations on the sphere and the disk.
\end{mainthm}

Previously, Gerber \cite{gerber_conditional_1985} proved the existence of analytic symplectomorphisms of these surfaces which are ergodic and with positive metric entropy (and so with infinitely many periodic points). Here, this theorem answer positively to the first part of a question from \cite[Question 5]{fayad_questions_2019} by constructing analytic ergodic pseudo-rotations on the disk.

\begin{mainthm}\label{thm:scheme_emer}
There exist analytic pseudo-rotations on the sphere and the disk with maximal order of local emergence.
\end{mainthm}

To obtain these results, we prove in \cref{sec:erg,sec:emer} that these properties are realizable through AbC schemes. These schemes and their proofs involve only $C^0$ and $C^1$ topologies, but the Principle implies that these properties are then satisfied by an analytic pseudo-rotation.\\

In this article, we start in \cref{sec:abc_scheme} by presenting AbC schemes which enable to build symplectomorphisms satisfying desired properties. These AbC schemes are defined in a $C^r$-topology for $0\leq r \leq +\infty$, in particular, when we show that such a scheme realizes a property, we only need to work in the $C^r$-topology. Then, we will recall the AbC Principle in \cref{thm:abc_princ} which provides analytic symplectomorphims satisfying a property realized by an AbC scheme. In particular, we can ask this analytic symplectomorphism to be a pseudo-rotation by \cref{prop:pseudo_rot}. Therefore, working in the $C^0$ and $C^1$ topologies will provide the analytic pseudo-rotations for \cref{thm:erg_scheme,thm:scheme_emer} thanks to the Principle.\\

Then, in \cref{sec:erg,sec:emer}, we prove that ergodicity or emergence are realizable by AbC schemes. In the Approximation by Conjugacy method we proceed by successive conjugations. The first step of the proofs is to construct symplectomorphisms conjugated to a rotation which satisfies the property up to a precision epsilon.

 In \cref{sec:erg}, we prove the $C^0$-AbC realization of ergodicity. First, we set up the construction of ergodicity on the cylinder from \cite{berger_analytic_2022} into a $C^0$-AbC scheme. Then, as ergodicity is a property invariant by area preserving conjugacy, we can push forward this AbC scheme on the sphere and the disk by a symplectomorphism from the interior of the cylinder onto a subset of full measure of the surface. Therefore we obtain the $C^0$-AbC realization of ergodicity on the sphere and the disk and we can use the AbC Principle to obtain \cref{thm:erg_scheme}.

Finally, in \cref{sec:emer}, we obtain the $C^1$-AbC realization of emergence. However, we can not proceed as easily as above because the emergence order is invariant by area preserving conjugacies which are \emph{bi-Lipschitz}, while the above projections from the cylinder to the disk or the sphere are \emph{not} bi-Lipschitz. To tackle this issue, we use the fact that \cite{berger_analytic_2022} constructed a sequence of conjugacies which are $C^0$-close to the identity. They conjugate rotations to dynamics which leave  invariant large compact subsets of the cylinder  on which the emergence is large. As   the projection is bi-Lipschitz on these compact subsets,  the push-forward of these conjugacies  and dynamics to the sphere and the disk displays 
high order of local emergence at infinitely small scale; so does their limit. We formalize this as a $C^1$-AbC scheme and then  we apply the AbC Principle to obtain \cref{thm:scheme_emer}. \\

\textbf{Acknowledgement.} {\it I am very thankful to Pierre Berger who introduced me to the Anosov Katok method through several papers, especially his own about the AbC Principle which motivated this article. I also thanks him for the supervision on this work and the several re-readings of this paper. Moreover, I would like to thanks Raphaël Krikorian for his encouragements on my work.}

\addcontentsline{toc}{section}{\bfseries{Notation}}
\section*{Notation}

\paragraph{Surfaces and spaces}\ \\

In this paper we will work on the following surfaces:

\begin{enumerate}
\item$\A := \T \times \I$ the cylinder, where $\T$ is the torus $\R / \Z$ and $\I = [-1;1]$.
\item $\Di := \lbrace x \in \R^2 ; x_1^2 + x_2^2 \leq 1 \rbrace$ the unit closed disk.
\item$\Sp := \lbrace x \in \R^3, x_1^2 +x_2^2 +x_3^2 =1 \rbrace$ the sphere. 
\end{enumerate}

The circle $\T$ acts canonically by rotation on the disk and the cylinder , it also acts by rotation of axis $x_3$ on the sphere. We denote by $R_\alpha$ such a rotation for $\alpha \in \T$.\\

We consider $\M \in \lbrace \A; \Sp; \Di \rbrace$ endowed with its canonical symplectic form $\Omega$. The form $\Omega$ defines a probability measure on the surface denoted $\Leb$. For $1 \leq r \leq +\infty$, we denote $\symp{r}{\M}$ the space of $C^r$-symplectomorphisms endowed with the $C^r$ topology. For $r=0$, $\symp{0}{\M}$ denotes the space of symplectic homeomorphisms, defined as the closure of $\symp{1}{\M}$ for the $C^0$ topology; in dimension 2 it coincides with the space of area-preserving homeomorphisms \cite[Theorem I]{oh_c0-coerciveness_2006}. Let $\symp{\omega}{\M}$ denote the set of analytic symplectomorphisms of $\M$.\\
When $0 \leq r < +\infty$, we denote by $d_{C^r}$ the distance inducing the $C^r$ topology.

We will work on these symplectic surfaces:

\begin{enumerate}
\item $\check{\A} := \T \times (-1;1)$,
\item $\check{\Di} := \croset{x \in \Di \ : \ 0< x^2_1 + x_2^2  <1}$,
\item$\check{\Sp} := \croset{x \in \Sp \ : \ x \neq (0,0,\pm 1)}$. 
\end{enumerate}

We may observe that these surfaces are symplectomorphic via

\begin{equation}\label{eq:pi}
\begin{array}{rrcl}
\pi : &\check{\A} &\rightarrow &\check{\M}\\
&(\theta,y) &\mapsto &\left\lbrace
\begin{array}{cl}
(\theta,y) &\text{if } \M = \A\\
\left( \sqrt{1-y^2}\cos(2\pi \theta) , \sqrt{1-y^2}\sin (2\pi \theta ),y \right) &\text{if } \M = \Sp\\
\sqrt{\frac{1+y}{2}} \left( \cos (2\pi \theta) , \sin(2\pi\theta) \right) &\text{if } \M = \Di
\end{array}\right.
\end{array}.
\end{equation}

We use the same notation $\pi$ for the three surfaces $\M \in \croset{\A; \Sp ; \Di}$. Note that $\pi$ might be extended to a surjective continuous map from $\A$ to $\M$.\\

Then, for $r \geq 0$, we can now define the following subspaces of $\symp{r}{\M}$:
\begin{itemize}
\item the subspace formed by the elements compactly supported in $\check{\M}$:
$$\sympc{r}{\M} = \croset{f \in \symp{r}{\M} \, : \, \overline{\mathrm{supp}(f)} \subset \check{\M}},$$
\item the subspace formed by elements equal to the identity on $\M \setminus \check{\M}$:
$$\sympo{r}{\M} =  \croset{f \in \symp{r}{\M} \, : \, \rest{f}{\M\setminus \check{\M}} = \id }.$$
\end{itemize}

We consider for $\eta \in [0;1)$ the surfaces
$$\A_\eta = \T \times \I_\eta \; \text{ and } \; \M_\eta = \pi(\A_\eta),$$
where $\I_\eta = (-1+\eta ; 1-\eta)$. In particular $\check{\M} = \M_0$, and for $f \in \sympc{r}{\M}$ we have some $\eta >0$ such that the support of $f$ is included in $\M_\eta$.

\paragraph{Kantorovich-Wasserstein distance}\ \\

We will work with the space of probability measures on $\M$, denoted by $\mathcal{M}(\M)$, endowed with the Kantorovich-Wassertein distance:
$$\forall \mu_1, \mu_2 \in \mathcal{M}(\M), \; d_K(\mu_1,\mu_2) := \inf \croset{\inint{\M^2}{} d(x_1,x_2) d\mu \ : \ \mu \in \mathcal{M}(\M^2) \text{ s.t. } {p_i}_*\mu = \mu_i \, \forall i \in \croset{1;2}}$$
 
where for $i \in \croset{1;2}$, $p_i : (x_1,x_2) \in \M^2 \mapsto x_i \in \M$ is the canonical projection. This distance induces the weak $\star$-topology on $\mathcal{M}(\M)$, which is compact. Moreover, by the duality theorem of Kantorovich and Rubinstein, we have the dual formulation:
$$d_K(\mu_1,\mu_2) = \sup \croset{\inint{\M}{}f d(\mu_1 - \mu_2) \; : \; f: \M \rightarrow \R \text{ 1-lipschitz}}.$$

We will give more properties on this distance in \cref{an:kanto}.\\
Given a symplectomorphism $f$ of $\M$, we consider the empirical measures for $x$ in $\M$ and $n \geq 1$:
$$e^f_n (x) := \frac{1}{n} \isum{k=1}{n}\delta_{f^k(x)}.$$

\paragraph{Emergence}\ \\

By Birkhoff's ergodic theorem, the following function is measurable and defined $\Leb$ almost everywhere on $\M$:
$$e^f : x \in \M \mapsto \lconv{n}{\infty} e^f_n(x) \in \mathcal{M}(\M).$$
It is called the \emph{empirical function} of $f$. Note that $e^f(x)$ describes the statistical behaviour of the orbit of $x$.\\

The \textit{ergodic decomposition} of $f$ is defined as the pushforward $\hat{e} := e_*^f \Leb$; it describes the statistical behaviour of orbits of $f$. This is a probability measure on the space of measures $\mathcal{M}(\M)$. Together with the Kantorovich distance, it allows to quantify a size of the space of invariant measures of $f$. This size called \textit{the order of local emergence} of $f$ is defined by: \label{def:emer}
$$\overline{O}\mathcal{E}_{loc}(f) = \limsup_{\epsilon \rightarrow 0} \int_{\mathcal{M}(\M)} \frac{\log \lvert \log \hat{e} \left( B(\mu,\epsilon) \right)\rvert}{\lvert \log \epsilon \rvert} d\hat{e}(\mu).$$
By \cite[Proposition 1.16]{helfter_scales_2024}, the order of local emergence is at most $2$ on a surface. We say that a symplectomorphism on $\M$ has \textit{maximal order of local emergence} if
$$\overline{O}\mathcal{E}_{loc}(f) = 2.$$

For more details about emergence, one can refer to \cite[§4]{berger_complexities_2019} for definitions. Several results on emergence are also obtained in \cite{berger_emergence_2021} and in \cite{berger_emergence_2017}.

\section{AbC Principle}\label{sec:abc_scheme}

In this section we present the AbC principle introduced by Berger in \cite{berger_analytic_2024}.\\

First we define what an AbC scheme is, where AbC stands for Approximation by Conjugacy. Let us denote by $\mathcal{T}^r$ the pull-back by the restriction $\rest{\cdot}{\check{\M}} : \symp{r}{\M} \mapsto \symp{r}{\check{\M}}$ of the compact-open $C^r$-topology on $\symp{r}{\check{\M}}$, for $\M \in \croset{\A ; \Sp; \Di}$ and $0\leq r \leq +\infty$. Where a basis of neighbourhoods of $f \in \symp{r}{\M}$ is:
$$\croset{g \in \symp{r}{\M} \, : \, d_{C^s} (\rest{f}{\M_\eta}, \rest{g}{\M_\eta}) < \epsilon },$$
taken among $\eta,\epsilon >0$ and finite $s \leq r$.

\begin{defin}[AbC scheme]\label{def:abc}
A \underline{$C^r$-AbC scheme} is a map:
$$(h,\alpha) \in \sympo{r}{\M} \times \Q/\Z \mapsto \left (U(h,\alpha), \nu(h,\alpha)\right ) \in \mathcal{T}^r \times (0,\infty)$$
such that:
\begin{enumerate}[label = \alph*)]
\item Each open set $U(h,\alpha)$ contains a map $\hat{h}\in \sympo{r}{\M}$  satisfying
$$ \iconj{\hat{h}}{R_\alpha} = \iconj{h}{R_\alpha}.$$

\item  Given any sequences $(h_n)_n \in \sympo{r}{\M}^\N$ and $(\alpha_n)_n \in \left (\Q / \Z \right )^\N$ satisfying :
$$h_0 = id \, , \;\; \alpha_0 = 0,$$
and for $n \geq 0$:
$$ \iconj{h_{n+1}}{R_{\alpha_n}} = \iconj{h_n}{R_{\alpha_n}},$$
$$h_{n+1} \in U(h_n,\alpha_n ) \, \text{ and } \, 0< \lvert \alpha_{n+1} - \alpha_n \rvert < \nu(h_{n+1}, \alpha_n )$$
then the sequence $f_n := \iconj{h_n}{R_{\alpha_n}}$ converges to a map $f$ in $\symp{r}{\M}$.
\end{enumerate}
The map $f$ is said \underline{constructed from the $C^r$-scheme} $(U,\nu)$.
\end{defin}
The aim of such AbC schemes is to construct maps which satisfy some property $(\mathcal{P})$:

\begin{defin}
A property ($\mathcal{P}$) on $\symp{r}{\M}$ is an \underline{AbC realizable $C^r$-property} if:
\begin{enumerate}
\item The property ($\mathcal{P}$) is invariant by $\symp{r}{\M}$-conjugacy: if $f \in \symp{r}{\M}$ satisfies ($\mathcal{P}$) then so does $\iconj{h}{f}$ for any $h \in \symp{r}{\M}$.
\item There exists a $C^r$-AbC scheme such that any map $f\in \symp{r}{\M}$ constructed from it satisfies ($\mathcal{P}$).
\end{enumerate}
\end{defin}

For instance, Berger proved that transitivity is an AbC $C^0$-realizable property. In Sections \ref{sec:erg} and \ref{sec:emer}, we will show that the following properties are realizable, as problematized by Berger in \cite[Problem 2.5]{berger_analytic_2024}.

\begin{prop}\label{prop:real_erg}
Being ergodic is an AbC realizable $C^0$-property.
\end{prop}

\begin{prop}\label{prop:real_emer}
Having maximal order of local emergence is an AbC realizable $C^1$-property.
\end{prop}

Yet, an AbC scheme also allows to construct pseudo-rotations, where a pseudo-rotations is a dynamic with a finite number of periodic points. In the cases of the cylinder, the disk or sphere this number must be respectively $0$, $1$, or $2$. In addition, these periodic points are all elliptic. 

\begin{prop}[\textnormal{\cite[Prop 2.4]{berger_analytic_2024}}]\label{prop:pseudo_rot}
For every $r \geq 1$, let $(\mathcal{Q})$ be the $C^r$-property of being a pseudo-rotation and let $(\mathcal{P})$ be any AbC realizable $C^r$-property. Then the conjunction $(\mathcal{P})\wedge (\mathcal{Q})$ is an AbC realizable $C^r$-property.
\end{prop}

Then, the use of open sets in the scheme allows to use density results, so that a realizable property is also realizable for stronger regularity as stated in the two next results.

\begin{prop}[\textnormal{\cite[Prop 2.6]{berger_analytic_2024}}]
For $+\infty \geq s \geq r \geq 0$, if $(\mathcal{P})$ is an AbC $C^r$-realizable property then it is an AbC $C^s$-realizable property.
\end{prop}

In addition of this regularization result, Pierre Berger obtain in \cite[Theorem D]{berger_analytic_2024} that any realizable property can be achieved analytically, as stated in his following main theorem.

\begin{theorem}[AbC principle]\label{thm:abc_princ}
For every $0 \leq r \leq \infty$, any $AbC$ realizable $C^r$-property $(\mathcal{P})$ is satisfied by a certain $f \in \symp{\omega}{\M}$.
\end{theorem}

With this principle we can obtain through the realization of ergodicity and maximal order of local emergence our first main theorems.

\begin{proof}[Proof of Theorem \ref{thm:erg_scheme}]
It is a direct consequence of the latter AbC principle, \cref{prop:real_erg} on realization of ergodicity, and \cref{prop:pseudo_rot}.
\end{proof}

\begin{proof}[Proof of Theorem \ref{thm:scheme_emer}]
It is a direct consequence of the latter AbC principle, \cref{prop:real_emer} on the realization of maximal order of local emergence, and \cref{prop:pseudo_rot}.
\end{proof}

\section{Ergodicity}\label{sec:erg}
We are going to achieve ergodicity through an AbC scheme in this section, i.e. we aim to prove \cref{prop:real_erg} which states that being ergodic is an AbC realizable $C^0$-property.\\

First observe that $\check{\M}$ is symplectomorphic to $\check{\A}$ through the symplectomorphism $\pi$
and $\check{\M}$ differs from $\M$ by a negligible sets of point. Thus we just need to prove the AbC realization of the ergodicity on $\A$.\\

To obtain this result, we first recall a construction of symplectomorphisms of $\A$ which are compactly supported in $\check{\A}$ and send $\Leb_{\T\times \lbrace y \rbrace}$ to a measure close to $\Leb_\A$ for most $y$. Such will be useful to build the AbC scheme.\\

More precisely we invoke the following lemma from \cite[Lemma 2.2]{berger_analytic_2022}.

\begin{lemma}\label{lemma:herg}
For every $q \in \N^*$ and $\epsilon >0$, there exists $h \in \sympc{\infty}{\A}$ such that:
\begin{itemize}
\item the symplectomorphism $h$ commutes with the rotation $R_{\frac{1}{q}}$ ,
\item for every $(\theta , y) \in \A_\epsilon$, the measure $h_*\Leb_{\T\times \lbrace y \rbrace}$ is $\epsilon$-close to $\Leb_\A$.
\end{itemize}
\end{lemma}

\begin{proof}[Sketch of proof]
The idea is to construct a symplectomorphism which moves and deforms boxes into other boxes of the same volume with a symplectic affine map extended to $\A$ using Moser's trick. The boxes are chosen such that most of the horizontal lines are sent to curves close to be equidistributed on $\A$. As depicts \cref{fig:erg}.

\begin{figure}[!h]
\centering
\includegraphics[scale=0.35]{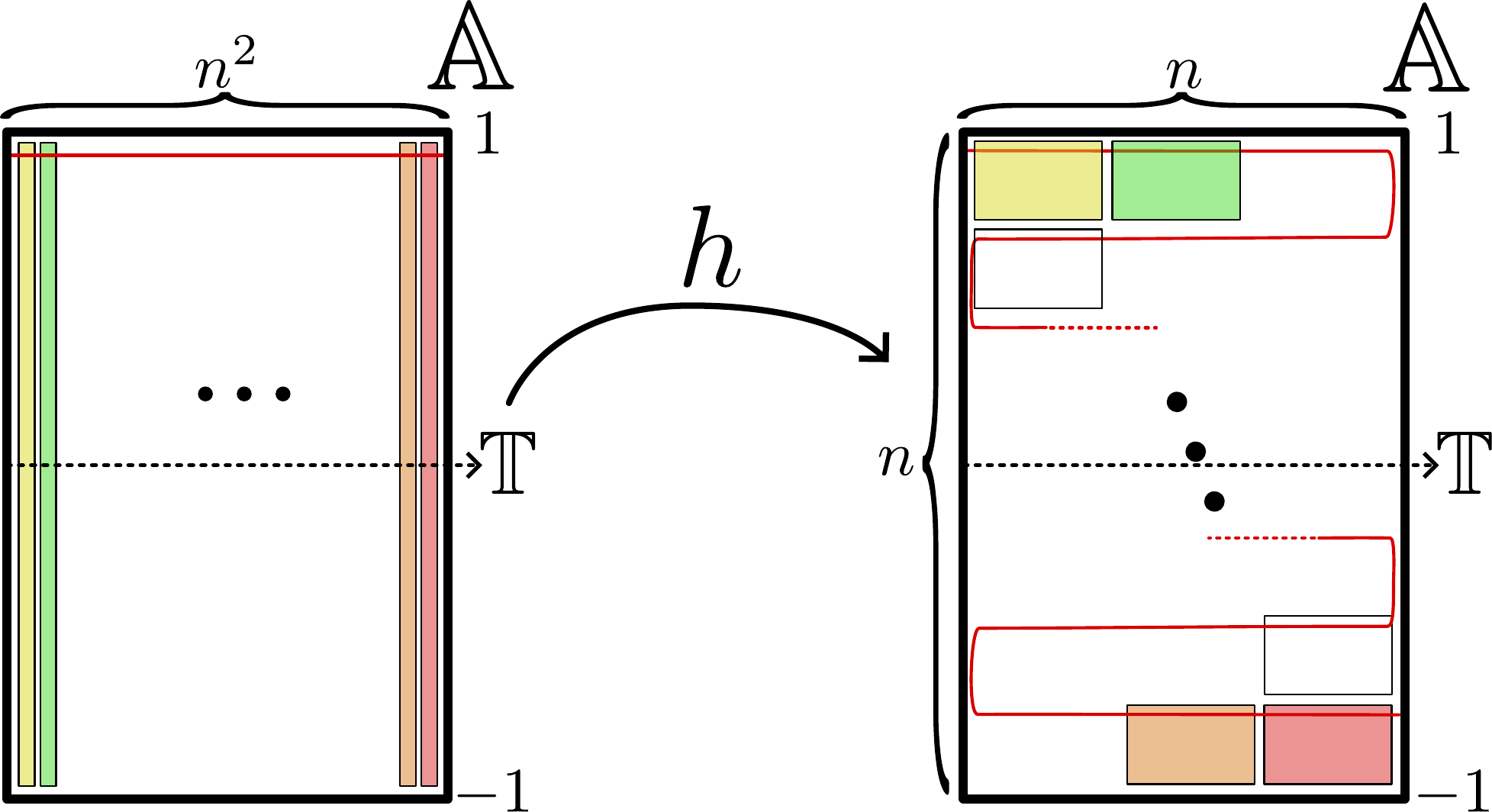}
\caption{Lemma for ergodicity}
\label{fig:erg}
\end{figure}

\end{proof}

Let us define a scheme which realizes ergodicity.\\

Let $\alpha = \frac{p}{q} \in \Q$, $h \in \sympo{0}{\A}$, $f = \iconj{h}{R_\alpha}$ and let $\epsilon >0$ be minimal such that for every $x$ in $h(\A_\epsilon)$, $e^f(x) = e^f_q(x)$ is $\epsilon$-close to $\Leb_\A$.\\

The following proposition will define a $C^0$-AbC scheme $(U,\nu)$ at the above elements $(h,\alpha)$.

\begin{prop}\label{prop:scheme_erg}
There is an open set $U(h,\alpha)$ in $\mathcal{T}^0$ containing a symplectomorphism $\hat{h}\in \sympo{0}{\A}$ such that $\iconj{\hat{h}}{R_\alpha} = f$. In addition, for every such $\hat{h}$ there exists $\nu(\hat{h},\alpha)>0$ s.t. for $0< \lvert \hat{\alpha} - \alpha \rvert < \nu(\hat{h},\alpha)$, the map $\hat{f} := \iconj{\hat{h}}{R_{\hat{\alpha}}}$ satisfies:
\begin{enumerate}[label = \arabic*)]
\item $d_K(e^{\hat{f}}(x), \Leb_\A) \leq \epsilon/2$ for every $x$ in $\hat{h}(\A_{\epsilon/2})$,
\item $d_{K,\infty}(e^f_k,e_k^{\hat{f}} ) := \sup_{x\in \A}(d_K(e_k^f(x),e_k^{\hat{f}}(x)) < \epsilon/2$ for every $k\leq q$,
\item $d_{C^0}(\hat{f},f) < \epsilon/2$,
\item if $\hat{\alpha}$ is rational, its denominator is greater than $q$.
\end{enumerate}
\end{prop}

\begin{proof}
Let $g \in \symp{1}{\A}$ be $\epsilon/16$-$C^0$ close to $h$ and $Q>1$ such that $g$ is $Q$-bi-Lipschitz. By the previous lemma, for $\eta \in (0,\tfrac{\epsilon}{16Q})$, there is a map $\widetilde{h}$ commuting with $R_{\frac{1}{q}}$ and satisfying that $\widetilde{h}_*\Leb_{\T\times \lbrace y \rbrace}$ is $\eta$-close to $\Leb_\A$ for every $y$ in $\I_\eta$.\\
Then define $\hat{h} = h \circ \widetilde{h}$. As $\widetilde{h}$ commutes with $R_\alpha$ we obtain
$$\iconj{\hat{h}}{R_\alpha} = f.$$

First, let us show that $\hat{h}_*\Leb_{\T\times \lbrace y \rbrace}$ is $\epsilon/8$-close to $\Leb_\A$ for every $y$ in $\I_{\epsilon/8}$. By \cref{prop:dk_c0} and as $g$ is $\epsilon/16$-$C^0$ close to $h$, we have for every $y\in I_\eta$ that 
$$d_K((h\circ \widetilde{h})_*\Leb_{\T\times \lbrace y \rbrace}, (g \circ \widetilde{h})_*\Leb_{\T\times \lbrace y \rbrace}) \leq \epsilon/16.$$
Moreover, we have by \cref{prop:kanto} and as $g$ is area preserving and $Q$-bi-Lipschitz that 
$$d_K((g \circ \widetilde{h})_*\Leb_{\T\times \lbrace y \rbrace}, \Leb_\A) \leq Qd_K(\widetilde{h}_*\Leb_{\T\times \lbrace y \rbrace}, \Leb_\A) \leq Q\eta \leq \frac{\epsilon}{16}.$$
Therefore, we obtain with the two latter inequalities that $\hat{h}_*\Leb_{\T\times \lbrace y \rbrace}$ is $\epsilon/8$-close to $\Leb_\A$ for every $y$ in $\I_{\epsilon/8}\subset \I_\eta$.\\

Since the measures $\hat{h}_*\Leb_{\T\times \lbrace y \rbrace}$ depend $C^0$-continuously on $\hat{h}$ uniformly on $y$ (\cref{prop:dk_c0}), there is a neighbourhood $U(h,\alpha)$ of $\hat{h}$ in $\mathcal{T}^0$ such that every $H$ in $U(h,\alpha)$ satisfies that $H_*\Leb_{\T\times \lbrace y \rbrace}$ is $\epsilon/4$ close to $\Leb_\A$ for every $y$ in $\I_{\epsilon/2}$.\\

Now let $\hat{h}$ be any element of $U(h,\alpha)$ satisfying $f = \iconj{\hat{h}}{R_\alpha}$. Let $(\theta,y) = \hat{h}^{-1}(x) \in \A_{\epsilon/2}$.\\
Given $\hat{\alpha} = \frac{\hat{p}}{\hat{q}}$, with $\hat{p}\wedge \hat{q} = 1$, and let $\hat{f} = \iconj{\hat{h}}{R_{\hat{\alpha}}}$, we have that
$$e^{\hat{f}}(x) = \frac{1}{\hat{q}} \sum_{k=1}^{\hat{q}} \delta_{\hat{h}(\theta + \frac{k\hat{p}}{\hat{q}},y)} = \hat{h}_*\left(\frac{1}{\hat{q}} \sum_{k=1}^{\hat{q}} \delta_{(\theta + \frac{k\hat{p}}{\hat{q}},y)}\right) .$$
Therefore when $\hat{q}$ is large, then the measure $\frac{1}{\hat{q}} \sum_{k=1}^{\hat{q}} \delta_{(\theta + \frac{k\hat{p}}{\hat{q}},y)}$ is uniformly close to $\Leb_{\T\times \lbrace y \rbrace}$ for every $y$ in $\I_{\epsilon/2}$. In particular, as $\hat{h}_*\Leb_{\T \times \croset{y}}$ is $\epsilon/4$ close to $\Leb_\A$ by construction of $U(h,\alpha)$, we obtain by \cref{coro:kanto} that for $\hat{q}$ large $e^{\hat{f}}(x)$ is $\epsilon/2$ close to $\Leb_\A$.\\
Hence we can define $\nu(\hat{h},\alpha)$ small enough so that every rational number $\hat{\alpha} = \frac{\hat{p}}{\hat{q}}$, satisfying $0 < \lvert \alpha - \hat{\alpha} \rvert < \nu(H,\alpha)
$, has its denominator $\hat{q}>q$ large enough, giving 4), such that:
\begin{itemize}
\item $e^{\hat{f}}(x)$ is $\epsilon/2$-close to $\Leb_\A$ for every $x$ in $\hat{h}(\A_{\epsilon/2})$.
\end{itemize}
Moreover, by taking $\nu(\hat{h},\alpha)$ small enough, we can ensure that $\hat{f}$ is sufficiently close to $f$ such that:
\begin{itemize}
\item $d_{K,\infty}(e^f_k,e_k^{\hat{f}} ) < \epsilon/2$ for $k\leq q$,
\item $d_{C^0}(f,\hat{f}) < \epsilon/2$.
\end{itemize}
Giving us 1), 2) and 3).
\end{proof}

Finally, by considering this scheme $(U,\nu)$, we are going to show that ergodicity is AbC $C^0$-realizable as stated in \cref{prop:real_erg}..

\begin{proof}[Proof of \cref{prop:real_erg}]
Let us prove that $(U,\nu)$ is a well-defined $C^0$-AbC scheme (see \cref{def:abc}).\\
For $(h,\alpha) \in \sympo{0}{\M} \times \Q/\Z$, \cref{prop:scheme_erg} ensures the existence of $\hat{h} \in U(h,\alpha)$ satisfying $\iconj{h}{R_\alpha}=\iconj{\hat{h}}{R_\alpha}$. Therefore Condition a) of \cref{def:abc} is satisfied.\\

To obtain Condition b), let us consider a sequence $f_n =\iconj{h_n}{R_{\alpha_n}}$ such that:
\begin{enumerate}[label = \arabic*)]
\item $\alpha_0 =0$ and $h_0 = id$,
\item $h_{n+1} \in U(h_n, \alpha_n)$ with $\iconj{h_{n+1}}{R_{\alpha_n}} = f_n$ and $0< \lvert \alpha_n - \alpha_{n+1} \rvert < \nu(h_{n+1},\alpha_n)$,
\end{enumerate}
and we write $\alpha_n = \frac{p_n}{q_n}$, with $p_n \wedge q_n = 1$.\\

Up to rescaling the distance on $\A$ we can assume that $\diam(\A) = \tfrac{1}{2}$, then we have that $e^{f_0}(x)$ is $\frac{1}{2}$-close to $\Leb_\A$ for every $x$ in $\A_{1/2}$. Then if we define for $n \in \N$:
\begin{equation}
\epsilon_n := \frac{1}{2^{n+1}},
\end{equation}
we obtain by \cref{prop:scheme_erg} the following properties for every $n\in \N^*$:
\begin{enumerate}[label = \Roman*)]
\item for every $x \in h_n(\A_{\epsilon_n})$ we have $d_K(e^{f_n}(x),\Leb_\A) \leq \epsilon_n$,
\item for every $k\leq q_{n-1}$ we have $\sup_{x\in \A}d_K(e^{f_n}_k(x),e^{f_{n-1}}_k(x)) \leq \epsilon_n$,
\item $d_{C^0}(f_n,f_{n-1})\leq \epsilon_n$,
\item and $q_n > q_{n-1}$.
\end{enumerate}
Therefore $(f_n)_n$ is a Cauchy sequence by III), hence converges toward some $f$ in $\symp{0}{\A}$, i.e. Condition b) of \cref{def:abc} is satisfied.\\

Now, let us show that $f$ is ergodic.\\

As $(f_n)$ converges to $f$, we have by II) and IV) for every $x \in \A$:
$$d_K(e^f_{q_n}(x), e^{f_n}(x)) \leq \sum_{m \geq n} d_K(e^{f_{m+1}}_{q_n}(x),e^{f_m}_{q_n}(x)) \leq \sum_{m \geq n} \epsilon_n = \epsilon_{n-1}.$$

Then it follows by I) and the latter equation:
\begin{equation}\label{eq:erg}
\forall x \in h_n(\A_{\epsilon_n}) \; : \; d_K(e^f_{q_n}(x), \Leb_\A) \leq \epsilon_{n-1} + \epsilon_n \leq \epsilon_{n-2}.
\end{equation}

Therefore we consider the set 
$$X = \displaystyle \bigcup_{n\geq 0} \bigcap_{k \geq n} h_k (\A_{\epsilon_k})$$
which is of full Lebesgue measure, as its complementary is a decreasing intersection of sets of measure lower than $\epsilon_{n-1}$. 
For $x$ in $X$ there exists $n \geq 0$ such that for every $k \geq n$, $x$ belongs to $h_k(\A_{\epsilon_k})$, therefore for every $k \geq n$ we have by \eqref{eq:erg}:
$$d_K(e_{q_k}^f(x),\Leb_\A) \leq \epsilon_{k-2}.$$
Hence we have: 
$$e_{q_k}^f(x) \underset{k \rightarrow +\infty}{\longrightarrow} \Leb_\A.$$\\
Up to reducing $X$ by intersecting it with the set of points $x$ where $(e_k^f(x))_k$ converges, which is also of full Lebesgue measure by Birkhoff. We obtain the ergodicity of $f$.\\

Finally, the well-known invariance by $\mathrm{Symp}^0$ conjugacy of ergodicity concludes the proof. 
\end{proof}

\section{Maximal order of local emergence}\label{sec:emer}

The purpose of this section is to prove the \cref{prop:real_emer} which states that having maximal order of local emergence is an AbC realizable $C^1$-property. To this end, we use a symplectomorphism $g_0$ built in \cite[Claim 2.6]{berger_analytic_2022} on the cylinder. The symplectomorphism $g_0$ sends a partition $(\T \times I_j)_j$ of $\A$ by a huge amount of horizontal strips onto a partition by subsets whose foliations by circles are in mean spaced apart. Then the flow $\iconj{g_0}{R_t}$ exhibits high emergence at a certain small scale $\eta$. However, when we push forward this construction by $\pi$ to the disk or the sphere to a map $g$, we encounter an issue. In fact, the emergence is a bi-Lipschitz invariant, but the partial derivatives of $\pi$ blow up near components of the boundary of $\A$.  To face this issue, we use the fact that $g_0$ can be chosen $C^0$-close to the identity. This will allow us to work on compact subsets $K$ and $g_0(K)$ of $int(\A)$ where $\pi$ is $Q$-bi-Lipschitz. Hence the flow $\iconj{g}{R_t}$ will exhibit high emergence on $\pi(g_0(K))$ at scale $\frac{\eta}{Q}$.\\

Next, since the symplectomorphism $g$ will be used to conjugate rotations, we introduce the following measures, whose pushforward by $g$ will be of particular interest: 
$$\forall y \in \I \; : \;\mu_y := \pi_* \Leb_{\T\times \croset{y}}.$$

Note that, when $\M = \Sp$ or $\Di$, $\mu_{\pm 1}$ are Dirac masses on the poles of the sphere, and $\mu_{-1}$ is a Dirac mass on the center of the disk. \\

Before going into the proof of the realization of emergence, we can show that emergence is invariant under $\symp{1}{\M}$-conjugacy as stated above. We recall that a symplectomorphism $f$ has a maximal order of local emergence if:
\begin{equation}\label{eq:def_emer}
 \overline{O}\mathcal{E}_{loc}(f) = \limsup_{\rconv{\epsilon}{0}} \int_{\mathcal{M}(\M)} \frac{\log\lvert \log \hat{e}(B(\mu,\epsilon))\rvert}{\lvert \log\epsilon \rvert}d\hat{e}(\mu) = 2.
\end{equation}
Where $\hat{e} = e^f_*\Leb_\M$ is the ergodic decomposition of $f$.

\begin{prop}\label{prop:inv_emer}
Having a maximal order of local emergence is invariant under $\symp{1}{\M}$-conjugacy.
\end{prop}

\begin{proof}
Let $f$ be a symplectomorphism of $\M$ such that $\overline{O}\mathcal{E}_{loc}(f) = 2$. Let $g \in \symp{1}{\M}$ and $Q>0$ such that $g$ is $Q$-bi-Lipschitz. Let us show that $\overline{O}\mathcal{E}_{loc}(\iconj{g}{f}) \geq \overline{O}\mathcal{E}_{loc}(f)$. Then, as $\overline{O}\mathcal{E}_{loc}(f)$ is maximal and equals $2$, we will deduce that $\overline{O}\mathcal{E}_{loc}(\iconj{g}{f}) =2$, i.e. the conjugacy has a maximal order of local emergence.\\ First, for $\Leb$ a.e. $x \in \M$, we have the following expression of the empirical function of $\iconj{g}{f}$:
$$e^{\iconj{g}{f}}(x) = \lconv{n}{\infty} \frac{1}{n}\isum{k=1}{n}\delta_{\iconj{g}{f^k}(x)} = \lconv{n}{\infty} g_* \left(\frac{1}{n}\isum{k=1}{n}\delta_{f^k(g^{-1}(x))}\right) = g_* e^f(g^{-1}(x)).$$

Then, if $\hat{e}_0 := {e^{\iconj{g}{f}}}_* \Leb$ denotes the ergodic decomposition of the conjugacy, we have for $\Leb$ a.e. $x \in \M$ by the latter equation and as $g$ is symplectic:
\begin{align*}
\hat{e}_0 (B(e^{\iconj{g}{f}}(x),\epsilon)) &= \Leb \croset{x' \in \M \, : \, d_K(e^{\iconj{g}{f}}(x),e^{\iconj{g}{f}}(x')) \leq \epsilon}\\
&=\Leb \croset{x' \in \M \, : \, d_K(g_*e^{f}(g^{-1}(x)),g_*e^{f}(x'))\leq \epsilon}
\end{align*}

It follows by \cref{prop:kanto} as $g$ is $Q$-bi-Lipschitz:
$$\hat{e}_0 (B(e^{\iconj{g}{f}}(x),\epsilon)) \geq \Leb \croset{x' \in \M \, : \, d_K(e^{f}(g^{-1}(x)),e^{f}(x')) \leq Q\epsilon} \geq  \hat{e} (B(e^{f}(g^{-1}(x)),Q\epsilon)).$$
Thus, we obtain with the expression of the order of local emergence \eqref{eq:def_emer} and the latter inequality:

\begin{align*}
\overline{O}\mathcal{E}_{loc}(\iconj{g}{f}) &=  \limsup_{\rconv{\epsilon}{0}} \int_{\M} \frac{\log\lvert \log \hat{e}_0(B(e^{\iconj{g}{f}}(x),\epsilon))\rvert}{\lvert \log\epsilon \rvert}d\Leb(x)\\
&\geq \limsup_{\rconv{\epsilon}{0}} \int_{\M} \frac{\log\lvert \log \hat{e} (B(e^{f}(g^{-1}(x)),Q\epsilon))\rvert}{\lvert \log \epsilon \rvert}d\Leb(x).
\end{align*}
Then, as $g$ is conservative we obtain:
$$\overline{O}\mathcal{E}_{loc}(\iconj{g}{f}) \geq \limsup_{\rconv{\epsilon}{0}} \frac{\lvert \log Q\epsilon \rvert}{\lvert \log \epsilon \rvert} \int_{\M} \frac{\log\lvert \log \hat{e} (B(e^{f}(x),Q\epsilon))\rvert}{\lvert \log Q\epsilon \rvert}d\Leb(x) = \overline{O}\mathcal{E}_{loc}(f) = 2.$$
Hence, the conjugacy $\iconj{g}{f}$ has a maximal order of local emergence, which concludes the proof.
\end{proof}

\subsection{Lemma for emergence}\label{sec:lem_em}
We first extract from \cite[Claim 2.6]{berger_analytic_2022} the following result on the cylinder.

\begin{lemma}\label{lemma:em_ber}
For every $q\in \N^*$ and $\epsilon_0 >0$ small enough, there exists $\eta_0>0$ arbitrary small and $g_0\in \sympc{\infty}{\A}$ such that:

\begin{enumerate}[label = \roman*)]
\item the symplectomorphism $g_0$ and its inverse are $\epsilon_0$-$C^0$ close to $\id$ and $g_0\circ R_{\frac{1}{q}} = R_{\frac{1}{q}} \circ g_0$,
\item and for every $y\in \I$ we have: 
$$\Leb_\I \croset{y' \in \I, d_K({g_0}_*\mesy{y},{g_0}_*\mesy{y'}) \leq 3\eta_0} \leq \exp(-\eta_0^{-2+\epsilon_0}).$$
\end{enumerate}

\end{lemma}

\begin{remark}
The measure $\Leb_\I$ is normalized to be a probability measure on $\I = [-1,1]$, i.e. for $[a,b] \subset \I$ we have $\Leb_\I [a,b] = \frac{b-a}{2}$. Hence we have $\Leb_\A = \Leb_\T \times \Leb_\I$.
\end{remark}

To obtain this result, Berger first uses his combinatorial argument with Bochi \cite[Prop. 4.2]{berger_emergence_2021} to obtain a construction of coloured necklaces on the cylinder, each necklace is composed of pearls which are box shaped, and colours are attributed to these pearls. Then, with Moser's trick, the necklace's pearls are sorted by color by a symplectomorphism $\widetilde{g}_0$ and the combinatorial argument ensures that the distances between images of two necklaces are lower bounded in mean, giving ii). Then, i) is obtained by replicating $\widetilde{g}_0$ on a tiling to obtain the wished symplectomorphism $g_0$.\\

Now, let us push forward this symplectomorphism by $\pi$ to obtain a similar result on $\M \in \croset{\A , \Sp, \Di}$. Yet, we will have to pay attention to the contraction of distances by $\pi$ to obtain ii) on $\M$. This will be done by using the fact that the symplectomorphism $g_0$ is $C^0$-close to $\id$. This will allow us to work on compact subsets of $int(\A)$ where $\pi$ is bi-Lipschitz. Then, let us recall the notations $\I_\epsilon := (-1+\epsilon, 1-\epsilon)$ and $\M_\epsilon := \pi(\T \times \I_\epsilon)$ for $\epsilon>0$, and observe that $\M_\epsilon$ is relatively compact in $\check{\M}$.

\begin{coro}\label{coro:symp_emer}
For every $\epsilon >0$ small enough and $q\in \N^*$, there exist $\eta>0$ arbitrary small and $g\in \sympc{\infty}{\M}$ such that:

\begin{enumerate}[label = \roman*)]
\item the symplectomorphism $g$ and its inverse are $\epsilon$-$C^0$ close to $\id$ and $g\circ R_{\frac{1}{q}} = R_{\frac{1}{q}} \circ g$,
\item and for every $y\in \I_\epsilon$ we have:
$$\Leb_\I \croset{y' \in \I, d_K(g_*\mu_y,g_*\mu_{y'}) \leq 3\eta} \leq \exp(-\eta^{-2+\epsilon}).$$
\end{enumerate}

\end{coro}

\begin{proof}
Let $\epsilon>0$ be small and $q \in \N^*$. Let $\epsilon_0 \in (0,\epsilon)$ be small, for $\eta_0>0$ arbitrary small there exists $g_0 \in \sympc{\infty}{\A}$ satisfying \cref{lemma:em_ber} for the parameters $\epsilon_0$, $q$ and $\eta_0$. We consider $g \in \sympc{\infty}{\M}$ to be the symplectomorphism which coincides with $\iconj{\pi}{g_0}$ on $\check{\M}$. First, $g$ does commute with $R_{1/q}$ by i) of \cref{lemma:em_ber}. Next, for $x \in \M$, if $x\in \M \setminus \check{\M}$ then $g(x) = x$ by definition of $\sympc{\infty}{\M}$. Otherwise, by i) of \cref{lemma:em_ber}, $g_0(\pi^{-1}(x))$ is $\epsilon_0$-close to $\pi^{-1}(x)$. Then, by uniform continuity of $\pi$ on $\A$, their images by $\pi$ are uniformly close when $\epsilon_0$ is small. Hence $g(x) = \iconj{\pi}{g_0}(x)$ is uniformly close to $x = \pi \circ \pi^{-1}(x)$ when $\epsilon_0$ is small. Likewise $g^{-1}$ is uniformly $C^0$-close to $\id$ when $\epsilon_0$ is small. Hence, by taking $\epsilon_0$ small enough, it proves i) of \cref{coro:symp_emer} with:
\begin{equation}\label{eq:g_id}
d_{C^0}(g,\id) \leq \frac{\epsilon}{16} \;\; \text{and} \;\; d_{C^0}(g^{-1},\id) \leq \frac{\epsilon}{16}.
\end{equation}

It remains to obtain ii).\\
First, let $Q>0$ be such that $\rest{\pi}{\A_{\epsilon/8}}$ is $Q$-bi-Lipschitz. We fix $y \in \I_\epsilon$. For $y' \in \I_{\epsilon/4}$, we have by \eqref{eq:g_id} that $g_*\mu_y$ and $g_*\mu_{y'}$ are supported in $\M_{\epsilon/8}$. Then we have by \cref{prop:kanto}, as $\pi$ is $Q$-bi-Lipschitz on $\A_{\epsilon/8}$:
\begin{equation*}
d_K({g_0}_*\mesy{y},{g_0}_*\mesy{y'}) \leq Qd_K(g_*\mu_y,g_*\mu_{y'}).
\end{equation*}
From this inequality we deduce the following inclusion:
\begin{equation}\label{eq:inc_eps}
\croset{y' \in \I_{\epsilon/4}, d_K(g_*\mu_y,g_*\mu_{y'}) \leq 3\frac{\eta_0}{Q}} \subset \croset{y' \in \I_{\epsilon/4}, d_K({g_0}_*\mesy{y},{g_0}_*\mesy{y'}) \leq 3\eta_0}
\end{equation}
Now let $y' \in \I\setminus \I_{\epsilon/4}$. We obtain by \eqref{eq:g_id} and \cref{prop:dk_c0} that:
$$d_K(g_*\mu_y,\mu_y) \leq \frac{\epsilon}{16} \;\; \text{and} \;\; d_K(g_*\mu_{y'},\mu_{y'}) \leq \frac{\epsilon}{16}.$$
Hence we obtain by triangular inequality:
\begin{equation}\label{eq:d_gmu}
d_K(g_*\mu_y, g_*\mu_{y'}) \geq d_K(\mu_y,\mu_{y'}) - \frac{\epsilon}{8}.
\end{equation}
Moreover, a simple computation of the derivative of $\pi$ on longitudes $\croset{\theta= \theta_0}$ (see \eqref{eq:pi} page \pageref{eq:pi} for the expression of $\pi$) allows to obtain:
\begin{equation*}\label{eq:pi_y}
d\left (\pi\left (\T\times \croset{y}\right ),\pi\left (\T\times \croset{y'}\right )\right ) \geq \tfrac{1}{4}\lvert y-y'\rvert \geq \frac{3\epsilon}{16}.
\end{equation*}
Then, using the 1-Lipschitz function
$x \mapsto \max(0,d(\pi(\T\times \croset{y}),\pi(\T\times \croset{y'})) - d(\pi(\T\times \croset{y}),x))$ which is greater than $\tfrac{3\epsilon}{16}$ on the support of $\mu_y$ by the last equation and equals $0$ on the support of $\mu_{y'}$, we obtain by Kantorovich's theorem:
$$d_K(\mu_y,\mu_{y'}) \geq \frac{3\epsilon}{16}$$
Then, \eqref{eq:d_gmu} becomes:
$$d_K(g_*\mu_y, g_*\mu_{y'}) \geq \frac{\epsilon}{16}.$$
Therefore, as $\eta_0$ is small, we can assume $\eta_0 < \frac{Q\epsilon}{48}$, which gives $d_K(g_*\mu_y, g_*\mu_{y'}) > \frac{3\eta_0}{Q}$ for every $y'\in \I\setminus \I_{\epsilon/4}$. i.e. we have:
$$\croset{y' \in \I, d_K(g_*\mu_y,g_*\mu_{y'}) \leq 3\frac{\eta_0}{Q}}=\croset{y' \in \I_{\epsilon/4}, d_K(g_*\mu_y,g_*\mu_{y'}) \leq 3\frac{\eta_0}{Q}}.$$
Hence we can improve \eqref{eq:inc_eps} into:
$$\croset{y' \in \I, d_K(g_*\mu_y,g_*\mu_{y'}) \leq 3\frac{\eta_0}{Q}} \subset \croset{y' \in \I_{\epsilon/4}, d_K({g_0}_*\mesy{y},{g_0}_*\mesy{y'}) \leq 3\eta_0}.$$
This gives by ii) of \cref{lemma:em_ber}:
$$\Leb_\I \left( \croset{y' \in \I, d_K(g_*\mu_y,g_*\mu_{y'}) \leq 3\frac{\eta_0}{Q}} \right) \leq \exp\left (-\eta_0^{-2+\epsilon_0}\right ).$$
Finally, if we define $\eta := \frac{\eta_0}{Q}$, as $\eta_0$ and $\epsilon_0$ are small, we obtain ii) of \cref{coro:symp_emer} with the latter equation.
\end{proof}

\subsection{AbC scheme}\label{sec:scheme_em}

Similarly to ergodicity, let us use the latter result to define an AbC scheme realizing the maximal order of local emergence.
Let $\alpha = \frac{p}{q} \in \Q$ and $h \in \sympo{1}{\M}$. We denote 
\begin{equation}
f := \iconj{h}{R_\alpha}.
\end{equation}
Let us define $U(h,\alpha)$ and $\nu(h,\alpha)$. First, we consider the following function of $h$ and $\epsilon'\in (0,2)$:
\begin{equation}\label{eq:eta_eps}
\eta(h,\epsilon') := \inf \croset{\eta' \in (0,1] \, : \, \forall y \in \I_{\epsilon'} \, , \, \Leb_\I \croset{y' \in \I, d_K(h_*\mu_y,h_*\mu_{y'}) \leq \eta'} \leq 3\exp(-{\eta'}^{-2+\epsilon'})}.
\end{equation}
The function $\eta$ provides the minimum scale $\eta'$ at which the flow $\iconj{h}{R_t}$ has high emergence.

\begin{fact}\label{fact:eta}
The function $\eta$ takes it values in $(0,1)$ and $\eta(h,\cdot)$ is a decreasing function.
\end{fact}
\begin{proof}
Let $\epsilon' \in (0,2)$. First, observe that the set in the definition of $\eta(h,\epsilon')$ is not empty and contains a neighbourhood of $1$. Indeed, we have for $y\in \I_{\epsilon'}$:
$$\Leb_\I\croset{y'  \in \I, d_K(h_*\mu_y,h_*\mu_{y'}) \leq 1} \leq 1 < 3\exp (-1) = 3\exp (- 1^{-2+\epsilon'}).$$

Since $h \in \sympo{1}{\M}$, let $Q>1$ be such that $h$ is $Q$-bi-Lipschitz. By \cref{prop:kanto}, we have for $y,y' \in \I$ that $d_K(h_*\mu_y, h_*\mu_{y'}) \leq  Qd_K(\mu_y,\mu_{y'})$, in particular, as one can check that $\pi$ is $1/2$-Hölder on longitude $\croset{\theta = \theta_0}$, we have $C>0$ such that $d_K(h_*\mu_y, h_*\mu_{y'}) \leq CQ\lvert y-y'\rvert^{1/2} $. Then $\Leb_\I \croset{y' \in \I, d_K(h_*\mu_y,h_*\mu_{y'}) \leq \eta'} \geq \Leb_\I [y-(\tfrac{\eta'}{CQ})^2,y+(\tfrac{\eta'}{CQ})^2] =(\tfrac{\eta'}{CQ})^2$, which is large compared to $\exp(-{\eta'}^{-2+\epsilon'})$. Hence we have $\eta(h,\epsilon') >0$.\\

Let us show that $\eta(h,\cdot)$ is a decreasing function. Let $0 < \epsilon_1 < \epsilon_2 <2$. For $\eta' \in (0,1)$, observe that $\exp(-{\eta'}^{-2+\epsilon_1}) < \exp(-{\eta'}^{-2+\epsilon_2})$ and $\I_{\epsilon_2} \subset \I_{\epsilon_1}$. Then the condition required on $\eta'$ in \eqref{eq:eta_eps} is weakened in $\eta(h,\epsilon_2)$ in comparison with $\eta(h,\epsilon_1)$. Therefore, we have $\eta(h,\epsilon_2) < \eta(h,\epsilon_1)$.
\end{proof}

Next, we set 
\begin{equation}
\epsilon := \frac{1}{4q}.
\end{equation}
As $\eta(h,\cdot)$ is a decreasing function, note that $\eta(h,1) \leq \eta(h,\epsilon)$.
\begin{prop}\label{prop:scheme_emer}
There exists an open set $U(h,\alpha) \in \mathcal{T}^1$ containing a symplectomorphism $\hat{h}$ satisfying $\iconj{\hat{h}}{R_\alpha} =f$. Moreover, for any $H \in U(h,\alpha)$ it holds:
\begin{enumerate}[label = \roman*)]
\item 2$\eta(H,\epsilon) \leq \eta(h,1)$,
\item The symplectomorphism $H$ and its inverse are respectively $\epsilon\cdot\eta(h,1)$-$C^0$ close to $h$ and its inverse on $\M_{\epsilon\cdot\delta}$, with $\delta := \exp(-\eta(h,1)^{-2+\epsilon})$.
\end{enumerate}
\end{prop}

\begin{proof}
First, let us construct $\hat{h}$ by composing $h$ with a map $g$ from \cref{coro:symp_emer}. Let $Q\geq1$ be such that $h$ is $Q$-bi-Lipschitz. Let $\epsilon' \in (0,\epsilon)$ and $\eta \in (0,\eta(h,1))$ be small. Let $g \in \sympc{\infty}{\M}$ be a symplectomorphism satisfying \cref{coro:symp_emer} for the parameters $\epsilon'$, $q$ and $\eta$. Then we define $\hat{h} := h \circ g \in \sympo{1}{\M}$.\\

By i) of \cref{coro:symp_emer}, we have $f = \iconj{\hat{h}}{R_\alpha}$. Moreover, as $\epsilon'$ is small we can assume that $\epsilon' \leq \tfrac{\epsilon\eta(h,1)}{2Q}$. Then, by i) of \cref{coro:symp_emer} and as $h$ is $Q$-bi-Lipschitz, it follows:
\begin{equation}\label{eq:close_h}
d_{C^0}(\hat{h},h) \leq \epsilon\frac{\eta(h,1)}{2} \;\; \text{ and } \;\; d_{C^0}(\hat{h}^{-1},h^{-1}) \leq \epsilon\frac{\eta(h,1)}{2}
\end{equation}
Moreover by ii) of \cref{coro:symp_emer} we have:
\begin{equation}\label{eq:em_I}
\forall y \in \I_{\epsilon'} \; , \; \Leb_\I \croset{y' \in \I \, : \, d_K(g_*\mu_y,g_*\mu_{y'}) \leq 3\eta} \leq \exp(-\eta^{-2+\epsilon'}).
\end{equation}

Hence we define $\hat{\eta} := \tfrac{\eta}{2Q} \leq \frac{\eta(h,1)}{2}$. Observe that, as $\epsilon'$ and $\eta$ are small, we have 
\begin{equation}\label{eq:hat_delta}
\hat{\delta} := \exp(-\hat{\eta}^{-2+\epsilon}) \geq 2\exp(-\eta^{-2+\epsilon'}).
\end{equation}
Hence, by \cref{prop:kanto} and \eqref{eq:em_I} we obtain:
\begin{equation}\label{eq:em_hath}
\forall y \in \I_\epsilon \; : \; \Leb_\I \croset{y' \in \I \, : \, d_K(\hat{h}_*\mu_y,\hat{h}_*\mu_{y'}) \leq 3\hat{\eta}} \leq \Leb_\I \croset{y' \in \I \, : \, d_K(g_*\mu_y,g_*\mu_{y'}) \leq 3\eta} \leq \frac{\hat{\delta}}{2}.
\end{equation}

Now, let us define $U(h,\alpha)$ as a neighbourhood of $\hat{h}$ in $\mathcal{T}^1$ to obtain i) and ii).\\

We choose $U(h,\alpha)$ to be a sufficiently small neighbourhood of $\hat{h}$ in $\mathcal{T}^1$ so that, for every $H$ in $U(h,\alpha)$, $H$ and its inverse are $\hat{\eta}\cdot\epsilon$-$C^0$ close to $\hat{h}$ and its inverse on $\M_{\epsilon\cdot\delta/2}$. In particular, by \eqref{eq:close_h}, $H$ satisfies ii) as $\hat{\eta}\leq \frac{\eta(h,1)}{2}$. Then we obtain by triangular inequality, \cref{prop:dk_c0}, and \eqref{eq:em_hath} that for every $y \in \I_\epsilon$:
$$\Leb_\I \croset{y' \in \I_{\epsilon\cdot \hat{\delta}/2 } \, : \, d_K(H_*\mu_y,H_*\mu_{y'}) \leq \hat{\eta}} \leq \Leb_\I \croset{y' \in \I \, : \, d_K(\hat{h}_*\mu_y,\hat{h}_*\mu_{y'}) \leq 3\hat{\eta}} \leq \frac{\hat{\delta}}{2}.$$
Next, note that $\Leb_\I( \I\setminus\I_{\epsilon\cdot\hat{\delta}/2 }) \leq \frac{\hat{\delta}}{2}$. Then the latter inequality becomes:
$$ \forall y \in \I_\epsilon \; : \;\Leb_\I \croset{y' \in \I \, : \, d_K(H_*\mu_y,H_*\mu_{y'}) \leq \hat{\eta}} \leq \frac{\hat{\delta}}{2} +  \Leb_\I( \I\setminus\I_{\epsilon\cdot\hat{\delta}/2 }) = \hat{\delta} \leq 3\hat{\delta}.$$
Then, by \eqref{eq:eta_eps} and \eqref{eq:hat_delta}, the latter inequality gives us $\eta(H,\epsilon) \leq \hat{\eta} \leq \frac{\eta(h,1)}{2}$. Thus we have that i) of \cref{prop:scheme_emer} is satisfied.
\end{proof}

Then, by defining $\nu(h,\alpha) >0$ small enough, we immediately obtain the following.

\begin{prop}\label{prop:scheme_emer_nu}
There exists $\nu(h,\alpha) >0$ such that for every $\hat{\alpha} \in (\alpha -3\nu(h,\alpha), \alpha + 3\nu(h,\alpha))$ it holds:
\begin{enumerate}[label= \roman*)]
\item  $d_{C^1}(\iconj{h}{R_{\hat{\alpha}}} , f)\leq \epsilon$,
\item if in addition $\hat{\alpha} \in \Q\setminus \croset{\alpha}$, then its denominator is greater than $4q$.
\end{enumerate}
\end{prop}

\subsection{Proof of the realization}

We recall that we have proved in \cref{prop:inv_emer} that having a maximal order of local emergence is invariant under $\symp{1}{\M}$-conjugacy. Then, to prove that having a maximal order of local emergence is a $C^1$-AbC realizable property, it remains to realize it with an AbC-$C^1$ scheme. The above subsection defines us a map:

$$\begin{array}{lccc}
(U,\nu): &\sympc{1}{\M} \times \Q/\Z &\rightarrow &\mathcal{T}^1 \times (0,\infty)\\
&(h,\alpha) &\mapsto &(U(h,\alpha),\nu(h,\alpha)).
\end{array}$$

Let us show that the map $(U,\nu)$ is a $C^1$-AbC scheme which realizes the property of having maximal order of local emergence. In particular, we prove \cref{prop:real_emer}.

\begin{proof}[Proof of \cref{prop:real_emer}]
First, let us show that \cref{prop:scheme_emer,prop:scheme_emer_nu} provide a well-defined AbC $C^1$-scheme $(U,\nu)$.\\

For $(h,\alpha) \in \sympc{1}{\M} \times \Q/\Z$, \cref{prop:scheme_emer} ensures the existence of $\hat{h} \in U(h,\alpha)$ satisfying $\iconj{h}{R_\alpha}=\iconj{\hat{h}}{R_\alpha}$. Therefore Condition a) of \cref{def:abc} is satisfied.\\

For Condition b), let $(h_n)_{n\in \N}$ and $(\alpha_n)_{n \in \N}$ be sequences constructed using the scheme $(U,\,\nu)$. i.e we have $h_0 = id$, $\alpha_0 = 0$, and, for $n \in \N$, the map $h_{n+1}$ belongs to $U(h_n,\alpha_n)$ and $0 < \lvert \alpha_n - \alpha_{n+1} \rvert < \nu(h_{n+1},\alpha_n)$. We also have for every $n$: 
$$f_n := \iconj{h_n}{R_{\alpha_n}}=\iconj{h_{n+1}}{R_{\alpha_n}},$$
and we write $\alpha_n = \frac{p_n}{q_n}$ with $p_n \wedge q_n =1$.\\

By considering the sequence $(\epsilon_n)_{n\in \N}$ defined by 
\begin{equation}
\forall n \in \N \, : \, \epsilon_n := \frac{1}{4q_n}.
\end{equation}
We have by \cref{prop:scheme_emer_nu}:
\begin{enumerate}[label = \Roman*)]
\item $q_{n} \geq 4q_{n-1}$, hence $\epsilon_n \leq \frac{\epsilon_{n-1}}{4} \leq \frac{1}{4^{n+1}}$,
\item $d_{C^1}(f_{n} , f_{n-1})\leq \epsilon_{n-1}$.
\end{enumerate}

Therefore, by I) and II), the sequence $(f_n)_n$ is a Cauchy sequence for $d_{C^1}$, hence it converges toward a map $f$ in $\symp{1}{\M}$. Therefore, Condition b) of \cref{def:abc} is satisfied by the scheme $(U,\nu)$ and $(U,\nu)$ is a well defined $C^1$-AbC scheme.\\

Now, let us show that $f$ exhibits a maximal order of local emergence, i.e. that the AbC scheme $(U,\nu)$ realizes the maximal order of local emergence.\\

First, we define sequences $(\eta_n)_{n\in \N^*}$ and $(\delta_n)_{n\in \N^*}$ by 
\begin{equation}
\forall n \in \N^* \, : \, \eta_n := \eta(h_n,\epsilon_{n-1}) \;\; \text{ and } \; \; \delta_n := \exp(-\eta_n^{-2+\epsilon_{n-1}}).
\end{equation}

By i) of \cref{prop:scheme_emer}, we have the following fact on these sequences:
\begin{fact}\label{fact:seq_eta}
The sequences $(\eta_n)_n$ and $(\delta_n)_n$ decrease and converge to $0$. Moreover we have for every $n \in \N^*$:
$$\eta_{n+1} \leq \frac{\eta(h_n,1)}{2} \leq \frac{\eta_n}{2}.$$
\end{fact}

\begin{proof}[Proof of \cref{fact:seq_eta}]
First, assuming the inequalities of the fact, we directly obtain that $(\eta_n)_n$ decreases and goes to $0$ as the $(\eta_n)$ are positive by \cref{fact:eta}. Then as $(\epsilon_n)_n$ also decreases to $0$ by I), then so does $(\delta_n)_n$.\\
Newt, let us show the inequalities. By i) of \cref{prop:scheme_emer}, as $h_{n+1} \in U(h_n,\alpha_n)$, we obtain $\eta_{n+1} \leq \frac{\eta(h_n,1)}{2}$. Then, as $\epsilon_n \leq 1$ by I), and $\eta(h_n,\cdot)$ is decreasing by \cref{fact:eta}, it comes that $\eta(h_n,1) \leq \eta(h_n,\epsilon_{n-1}) = \eta_n$. Therefore we obtain
$$\eta_{n+1} \leq \frac{\eta(h_n,1)}{2} \leq \frac{\eta_n}{2}.$$
\end{proof}

With these sequences, we obtain the following fact by ii) of \cref{prop:scheme_emer}:

\begin{fact}\label{fact:hn}
Let $n\in \N^*$, the maps $h_{n+1}$ and its inverse are respectively $\epsilon_n\cdot\eta_n$-$C^0$ close to $h_n$ and its inverse on $\M_{\epsilon_{n}\cdot\delta_{n}}$.
\end{fact}

\begin{proof}[Proof of \cref{fact:hn}]
For $n\in \N^*$, as $h_{n+1} \in U(h_n,\alpha_n)$, we have by ii) of \cref{prop:scheme_emer} that $h_{n+1}$ and its inverse are respectively $\epsilon_n\cdot\eta(h_n,1)$-$C^0$ close to $h_n$ and its inverse on $\M_{\epsilon_{n}\cdot\delta}$, with $\delta := \exp(-\eta(h_n,1)^{-2+\epsilon_n})$. Yet, by \cref{fact:seq_eta} we have $\eta_n \geq \eta(h_n,1)$. Then $h_{n+1}$ and its inverse are respectively $\epsilon_n\cdot\eta_n$-$C^0$ close to $h_n$ and its inverse on $\M_{\epsilon_{n}\cdot\delta}$. Moreover $\eta_n \geq \eta(h_n,1)$ implies that $\delta_n \geq \delta$, i.e. $\M_{\epsilon_n\cdot\delta_n} \subset \M_{\epsilon_n\cdot\delta}$, which allows to conclude the proof.
\end{proof}

With \cref{fact:seq_eta,fact:hn}, $(h_n)_n$ and $(h_n^{-1})_n$ converge in the $C^0$ topology of $\check{\M}$. We denote by $h \in \symp{0}{\check{\M}}$ the limit of $(h_n)_n$. Then we want to express $f$ as the conjugacy of a rotation by $h$, so we prove the following fact.

\begin{fact}\label{fact:alpha}
The sequence $(\alpha_n)_n$ converges to $\alpha \in \R\setminus \Q$.\end{fact}

\begin{proof}[Proof of \cref{fact:alpha}]
First, $(\alpha_n)_n$ is a Cauchy sequence as ii) of \cref{prop:scheme_emer_nu} imposes $\lvert \alpha_n - \alpha_{n+1} \rvert < \tfrac{1}{3q_n}$ and $4q_n \leq q_{n+1}$, let $\alpha$ be its limit. Let us show that $\alpha$ lives in the nested intersection of the intervals 
\begin{equation}\label{eq:In}
I_n := [\alpha_n - 2\nu(h_{n+1},\alpha_n), \alpha_n + 2\nu(h_{n+1},\alpha_n)].
\end{equation}
In which case, as in $I_n$ every rational has a denominator greater than $q_n$ by ii) of \cref{prop:scheme_emer_nu}, $\alpha$ must be irrational.\\
First, by definition of the sequence $(\alpha_n)_n$ we have $d(\alpha_n,\alpha_{n+1}) \leq \nu(h_{n+1},\alpha_{n})$. Moreover, ii) of \cref{prop:scheme_emer_nu} imposes $d(\alpha_n,\alpha_{n+1}) \geq 3\nu(h_{n+2},\alpha_{n+1})$ as $q_n \leq \tfrac{1}{4} q_{n+1}$ and every rational number different from $\alpha_{n+1}$ in $(\alpha_{n+1} -3\nu(h_{n+2},\alpha_{n+1}),\alpha_{n+1} +3\nu(h_{n+2},\alpha_{n+1}))$ has its denominator greater than $4q_{n+1}$. Therefore we have:
\begin{equation}\label{eq:d_alphn}
\nu(h_{n+1},\alpha_{n}) \geq d(\alpha_n,\alpha_{n+1}) \geq 3\nu(h_{n+2},\alpha_{n+1}).
\end{equation}
Then $I_{n+1} \subset I_n$ by the latter equation and \eqref{eq:In} (see \cref{fig:In}). In particular we obtain for every $n$ that $(\alpha_m)_{m\geq n}$ lives in the closed interval $I_n$ and converges to $\alpha$, hence $\alpha \in \cap_{n}I_n$ which concludes the proof.

\begin{figure}[!h]
\centering
\includegraphics[scale=0.35]{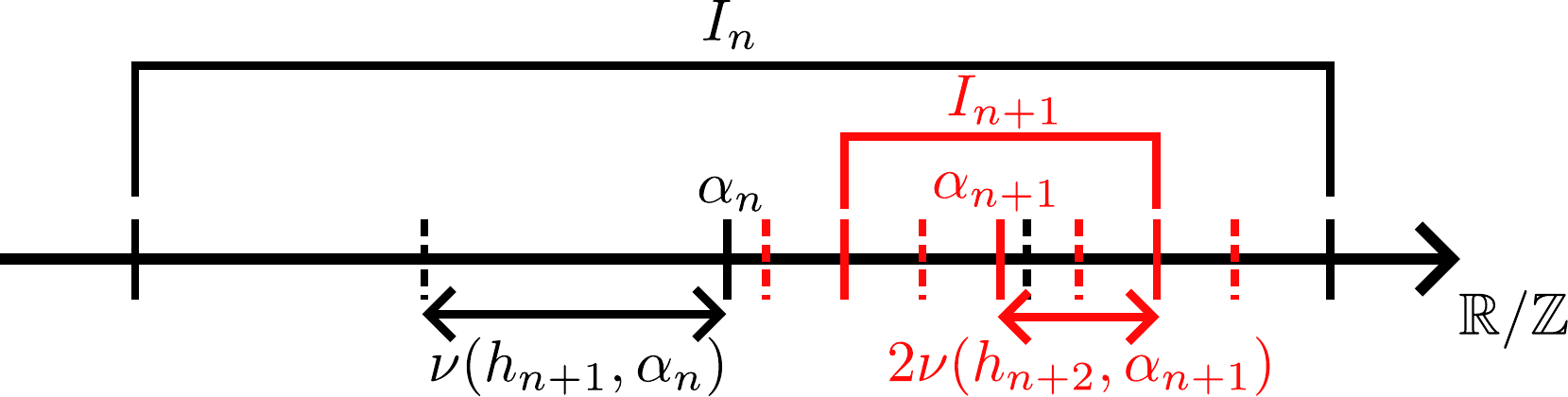}
\caption{Intervals $I_n$}
\label{fig:In}
\end{figure}

\end{proof}

By the latter fact and the convergence of $(h_n)_n$ and $(h_n^{-1})_n$, we obtain the following expression of $f$ on $\check{\M}$:
$$\rest{f}{\check{\M}} = \iconj{h}{R_\alpha}.$$

Then, for any $x = h (\pi(\theta,y)) \in \check{\M}$ we have:
\begin{equation}\label{eq:mes_emp}
 e^f(x) = h_*\mu_y.
\end{equation}

We prove below:
\begin{fact}\label{fact:emer}
For $n \in \N^*$, we have for every $x \in h(\M_{\epsilon_{n-1}})$:
$$\Leb_\M \croset{x' \in \M \, : \, d_K(e^f(x),e^f(x'))\leq \frac{\eta_n}{2}}  \leq (3+\epsilon_{n})\delta_{n}.$$
\end{fact}

With this fact, we can compute the order of local emergence of $f$:
\begin{align*}
\displaystyle \int_{\mathcal{M}(\M)} \frac{\lvert\log \lvert \log \hat{e}\left( B(\mu,\tfrac{\eta_n}{2}) \right) \rvert \rvert}{\lvert \log \tfrac{\eta_n}{2} \lvert} d \hat{e}(\mu) &= \displaystyle \int_{\check{\M}} \frac{\log \lvert \log \Leb_\M \croset{x' \in \M, \; d_K(e^f(x),e^f(x') ) \leq \tfrac{\eta_n}{2}} \rvert }{\lvert \log \tfrac{\eta_n}{2} \rvert} d \Leb_\M(x)\\
&\geq \int_{h(\M_{\epsilon_{n-1}})} \frac{\log \lvert \log((3+\epsilon_{n})\delta_{n})\rvert}{-\log(\tfrac{\eta_n}{2})} d \Leb_\M(x)\\
&\geq  \left( 1- \epsilon_{n-1} \right ) \frac{\log \left\lvert \eta_{n}^{-2+\epsilon_{n-1}}-\log(3+\epsilon_n) \right\rvert}{-\log(\eta_n/2)} \lequiv{n}{\infty} 2.
\end{align*}

This gives that the order of local emergence of $f$ is 2, i.e. is maximal.\\

Finally, as \cref{prop:inv_emer} shows the invariance of having a maximal order of local emergence by $\symp{1}{\M}$-conjugacy, we conclude that having a maximal order of local emergence is a $C^1$-AbC realizable property.
\end{proof}

\begin{proof}[Proof of \cref{fact:emer}]
First, let us show the following inequality for every $n \in \N^*$ and $x = h (\pi(\theta,y)) \in h(\M_{\epsilon_{n-1}})$:

\begin{equation}\label{eq:mes_leb1}
\Leb_\M \croset{x' \in h(\M_{\epsilon_n\cdot\delta_{n}})\, : \, d_K(e^f(x),e^f(x'))\leq \frac{\eta_n}{2}}  \leq 3\delta_{n}.
\end{equation}
Assuming this inequality, since $\Leb_\M ( h(\M_{\epsilon_n\cdot\delta_{n}}) ) = 1-\epsilon_n\cdot\delta_{n}$ and $(\delta_m)_m$ is decreasing, we can prove \cref{fact:emer}:
$$\Leb_\M \croset{x' \in \M \, : \, d_K(e^f(x),e^f(x'))\leq \frac{\eta_n}{2}}  \leq 3\delta_{n} + \epsilon_n\cdot\delta_{n} = (3+\epsilon_{n})\delta_{n}.$$

Then, by \eqref{eq:mes_emp} and as $\Leb_\M \pi(\T \times J) =\Leb_\I J$, the inequality \eqref{eq:mes_leb1} is equivalent to the following inequality for $y \in \I_{\epsilon_{n-1}}$:

\begin{equation}\label{eq:mes_leb2}
\Leb_\I \croset{y' \in \I_{\epsilon_n\cdot\delta_{n}}\, : \, d_K(h_*\mu_y, h_*\mu_{y'})\leq \frac{\eta_n}{2}}   \leq 3\delta_{n}.
\end{equation}

Yet, as $\eta_n = \eta(h_n,\epsilon_{n-1})$, we have by \eqref{eq:eta_eps}:

\begin{equation}\label{eq:em_hn}
\forall y \in \I_{\epsilon_{n-1}} \, : \,\Leb_\I\croset{y' \in \I \, : \, d_K({h_n}_*\mu_y,{h_n}_*\mu_{y'}) \leq \eta_n} \leq 3\exp(-\eta_n^{-2+\epsilon_{n-1}}) =  3\delta_{n}.
\end{equation}
Therefore, to prove \cref{eq:mes_leb2} with \cref{eq:em_hn}, we want to prove the following inclusion for $y \in \I_{\epsilon_{n-1}}$:
\begin{equation}\label{eq:incl_em}
\croset{y' \in \I_{\epsilon_n\cdot\delta_{n}}\, : \, d_K(h_*\mu_y, h_*\mu_{y'})\leq \frac{\eta_n}{2}} \subset \croset{y' \in \I \, : \, d_K({h_n}_*\mu_y,{h_n}_*\mu_{y'}) \leq \eta_n}.
\end{equation}
Then, let $n\in \N^*$ and $y$ be an element of $\I_{\epsilon_{n-1}}$.  For $y' \in \I_{\epsilon_n\cdot\delta_{n}}$ such that $d_K(h_*\mu_y, h_*\mu_{y'}) \leq \frac{\eta_n}{2}$, let us show that $d_K({h_{n}}_*\mu_y, {h_n}_*\mu_{y'}) \leq \eta_n$. First, we have by triangular inequality and $C^0$-convergence of $(h_m)_m$ to $h$ on $\check{\M}$:
$$d_K({h_{n}}_*\mu_y, {h_n}_*\mu_{y'}) \leq  \sum_{m\geq n}\left (d_K({h_m}_*\mu_y,{h_{m+1}}_*\mu_y) + d_K({h_m}_*\mu_{y'},{h_{m+1}}_*\mu_{y'})\right ) + d_K(h_*\mu_y, h_*\mu_{y'}).$$
Then, as $(\epsilon_m\cdot\delta_m)_m$ is decreasing and $\mu_y$ and $\mu_{y'}$ are supported in $\M_{\epsilon_n\cdot\delta_n}$, they are supported in $\M_{\epsilon_m\cdot\delta_m}$ for every $m\geq n$. Therefore we obtain by \cref{prop:dk_c0} and the latter equation:

$$d_K({h_{n}}_*\mu_y, {h_n}_*\mu_{y'}) \leq 2\sum_{m\geq n}\left (\sup_{z \in \M_{\epsilon_m\cdot\delta_{m}}} d(h_{m+1}(z),h_m(z))\right ) + d_K(h_*\mu_y, h_*\mu_{y'}).$$
Then, we obtain by hypothesis on $d_K(h_*\mu_y, h_*\mu_{y'})$, \cref{fact:hn}, and \cref{fact:seq_eta} that:
$$d_K({h_{n}}_*\mu_y, {h_n}_*\mu_{y'}) \leq 2\sum_{m\geq n} \epsilon_m\cdot\eta_m + \frac{\eta_n}{2} \leq 2\eta_n\sum_{m\geq n} \epsilon_m+ \frac{\eta_n}{2}.$$
Finally, as by I) we have $\sum_m\epsilon_m \leq 1/4$, it follows:
$$d_K({h_{n}}_*\mu_y, {h_n}_*\mu_{y'}) \leq \eta_n.$$
Therefore, we obtain the inclusion \eqref{eq:incl_em}. Then by \eqref{eq:incl_em} and \eqref{eq:em_hn} we obtain \eqref{eq:mes_leb2}. Hence we also obtain \eqref{eq:mes_leb1} which concludes the proof.

\end{proof}

\appendix

\section{Results on the Kantorovich distance}\label{an:kanto}
In this appendix, we present several results on the Kantorovich distance which are used in \cref{sec:erg,sec:emer}.\\
First, we have the following result from \cite[Prop. 1.3]{berger_analytic_2022} on the push-forward of measures by $C^0$ maps.
\begin{prop}\label{prop:dk_c0}
Let $X$ and $Y$ be measurable compact metric spaces. For $f_1,f_2 \in C^0(X,Y)$ and $\mu \in \mathcal{M}(X)$, the Kantorovich distance satisfies:
$$d_K({f_1}_*\mu,{f_2}_*\mu) \leq d_{C^0}(f_1,f_2).$$
\end{prop} 

Next, let us present a result on the Kantorovich distance for bi-Lipschitz maps which is used several times in \cref{sec:emer}.

\begin{prop}\label{prop:kanto}
Let $(X,d_X)$ and $(Y,d_Y)$ be measurable metric spaces and let $\phi : X \rightarrow Y$ be an inversible $Q$-bi-Lipschitz map. i.e. $\phi$ and $\phi^{-1}$ are $Q$-Lipschitz. Then for $\mu_1$ and $\mu_2$ two measures on $X$ we have:
$$\frac{1}{Q} d_K(\mu_1,\mu_2) \leq d_K(\phi_*\mu_1,\phi_*\mu_2) \leq Q d_K(\mu_1,\mu_2).$$
\end{prop}

\begin{proof}
To obtain the right-hand inequality, we consider $f:Y \rightarrow \R$ a $1$-Lipschitz map, then $\tfrac{1}{Q}f\circ \phi$ is also $1$-Lipschitz and we have by Kantorovich's theorem :
$$\int_Yfd(\phi_*\mu_1-\phi_*\mu_2) = \int_X f\circ \phi d(\mu_1-\mu_2) = Q\int_X\frac{1}{Q}f\circ \phi d(\mu_1-\mu_2) \leq Q d_K(\mu_1,\mu_2).$$
Which gives the right-hand inequality by Kantorovich's theorem. Then, we obtain the left-hand inequality with the right one by replacing $\phi$ by $\phi^{-1}$ and $\mu_i$ by $\phi_*\mu_i$ for $i\in \croset{1,2}$.
\end{proof}

Then for $\M \in \croset{\A , \Di, \Sp}$, we obtain as a corollary of these two proposition and the density of $\symp{1}{\M}$ in $\symp{0}{\M}$ that, for $h \in \symp{0}{\M}$, the map $\mu \in \mathcal{M}(\M) \mapsto h_\star \mu$ is uniformly continuous.

\begin{coro}\label{coro:kanto}
Let $h$ be in $\symp{0}{\M}$, then the map
$$\mu \in \mathcal{M}(\M) \mapsto h_*\mu \in \mathcal{M}(\M)$$
is uniformly continuous.
\end{coro}

\bibliographystyle{alpha} 
\bibliography{biblio}

\begin{thebibliography}{BCLR06}

\bibitem[AK70]{anosov_new_1970}
Dmitry Anosov and Anatole Katok.
\newblock New examples in smooth ergodic theory, ergodic diffeomorphisms.
\newblock {\em Trans. Mosc. Math. Soc}, 23(1):3--36, 1970.

\bibitem[BB21]{berger_emergence_2021}
Pierre Berger and Jairo Bochi.
\newblock On emergence and complexity of ergodic decompositions.
\newblock {\em Advances in Mathematics}, 390:107904, October 2021.

\bibitem[BCLR06]{beguin_pseudo-rotations_2006}
F.~Béguin, S.~Crovisier, and F.~Le~Roux.
\newblock Pseudo-rotations of the open annulus.
\newblock {\em Bulletin of the Brazilian Mathematical Society}, 37(2):275--306,
  June 2006.

\bibitem[Ber17]{berger_emergence_2017}
Pierre Berger.
\newblock Emergence and non-typicality of the finiteness of the attractors in
  many topologies.
\newblock {\em Proceedings of the Steklov Institute of Mathematics},
  297(1):1--27, May 2017.

\bibitem[Ber19]{berger_complexities_2019}
Pierre Berger.
\newblock Complexities of differentiable dynamical systems.
\newblock {\em Proceedings of the International Congress on Mathematical
  Physics, special issue of Journal of Mathematical Physics}, 2019.

\bibitem[Ber22]{berger_analytic_2022}
Pierre Berger.
\newblock Analytic pseudo-rotations.
\newblock {\em ArXiv, to appear Annals of Mathematics}, October 2022.

\bibitem[Ber24]{berger_analytic_2024}
Pierre Berger.
\newblock Analytic pseudo-rotations {II}: a principle for spheres, disks and
  annuli.
\newblock {\em ArXiv}, April 2024.

\bibitem[Bir41]{birkhoff_unsolved_1941}
George~D. Birkhoff.
\newblock Some {Unsolved} {Problems} of {Theoretical} {Dynamics}.
\newblock {\em Science}, 94(2452):598--600, December 1941.

\bibitem[BK19]{banerjee_real-analytic_2019}
Shilpak Banerjee and Philipp Kunde.
\newblock Real-analytic {AbC} constructions on the torus.
\newblock {\em Ergodic Theory and Dynamical Systems}, 39(10):2643--2688,
  October 2019.

\bibitem[FK04]{fayad_constructions_2004}
Bassam Fayad and Anatole Katok.
\newblock Constructions in elliptic dynamics.
\newblock {\em Ergodic Theory and Dynamical Systems}, 24(5):1477--1520, October
  2004.

\bibitem[FK14]{fayad_analytic_2014}
Bassam Fayad and Anatole Katok.
\newblock Analytic uniquely ergodic volume preserving maps on odd spheres.
\newblock {\em Commentarii Mathematici Helvetici}, 89(4):963--977, November
  2014.

\bibitem[FK19]{fayad_questions_2019}
Bassam Fayad and Raphaël Krikorian.
\newblock Some questions around quasi-periodic dynamics.
\newblock In {\em Proceedings of the {International} {Congress} of
  {Mathematicians} ({ICM} 2018)}, pages 1909--1932, Rio de Janeiro, Brazil, May
  2019. WORLD SCIENTIFIC.

\bibitem[Fur61]{furstenberg_strict_1961}
H.~Furstenberg.
\newblock Strict {Ergodicity} and {Transformation} of the {Torus}.
\newblock {\em American Journal of Mathematics}, 83(4):573--601, 1961.
\newblock Publisher: The Johns Hopkins University Press.

\bibitem[Ger85]{gerber_conditional_1985}
Marlies Gerber.
\newblock {\em Conditional {Stability} and {Real} {Analytic} {Pseudo}-{Anosov}
  {Maps}}.
\newblock American Mathematical Soc., 1985.
\newblock Google-Books-ID: qTjUCQAAQBAJ.

\bibitem[Hel24]{helfter_scales_2024}
Mathieu Helfter.
\newblock Scales.
\newblock {\em ArXiv}, May 2024.

\bibitem[Her98]{herman_open_1998}
Michael Herman.
\newblock Some open problems in dynamical systems.
\newblock In {\em Proceedings of the {International} {Congress} of
  {Mathematicians}}, volume~2, pages 797--808. Berlin, 1998.

\bibitem[Oh06]{oh_c0-coerciveness_2006}
Yong-Geun Oh.
\newblock \${C}{\textasciicircum}0\$-coerciveness of {Moser}'s problem and
  smoothing area preserving homeomorphisms.
\newblock {\em ArXiv}, December 2006.

\end{thebibliography}
\addcontentsline{toc}{section}{\bfseries{Bibliography}}
\end{document}